\documentclass{amsart}

\usepackage{graphicx}
\usepackage{amssymb}

\newtheorem{theorem}{Theorem}[section]
\newtheorem{lemma}[theorem]{Lemma}

\theoremstyle{definition}
\newtheorem{definition}[theorem]{Definition}

\numberwithin{equation}{section}
\numberwithin{figure}{section}

\begin{document}

\title{Anonymity in Predicting the Future}
\author{Dvij Bajpai and Daniel J.\ Velleman}
\date{}

\begin{abstract}
Consider an arbitrary set $S$ and an arbitrary function $f : \mathbb{R} \to S$.  We think of the domain of $f$ as representing time, and for each $x \in \mathbb{R}$, we think of $f(x)$ as the state of some system at time $x$.  Imagine that, at each time $x$, there is an agent who can see $f \upharpoonright (-\infty, x)$ and is trying to guess $f(x)$---in other words, the agent is trying to guess the present state of the system from its past history.  In a 2008 paper, Christopher Hardin and Alan Taylor use the axiom of choice to construct a strategy that the agents can use to guarantee that, for every function $f$, all but countably many of them will guess correctly.  In a 2013 monograph they introduce the idea of anonymous guessing strategies, in which the agents can see the past but don't know where they are located in time.  In this paper we consider a number of variations on anonymity.  For instance, what if, in addition to not knowing where they are located in time, agents also do not know the rate at which time is progressing? What if they have no sense of how much time elapses between any two events? We show that in some cases agents can still guess successfully, while in others they perform very poorly.
\end{abstract}

\maketitle

\section{Introduction}

In \cite{HTpeculiar}, Christopher Hardin and Alan Taylor present a remarkable consequence of the axiom of choice.  Consider an arbitrary set $S$ and a function $f : \mathbb{R} \to S$.  We think of the domain of $f$ as representing time, and for each $x \in \mathbb{R}$, we think of $f(x)$ as the state of some system at time $x$.  The question Hardin and Taylor address is whether or not it is possible to predict the state of the system from its past behavior.  In other words, can we predict $f(x)$ from $f \upharpoonright (-\infty, x)$?  If there is no restriction on the function $f$, then it seems that $f \upharpoonright (-\infty, x)$ gives us no information about $f(x)$, and therefore it is hard to see how such a prediction could be possible.  Nevertheless, using the axiom of choice, Hardin and Taylor produce a prediction strategy with the property that for every function $f$, the prediction is correct for all but countably many values of $x$.

It may be helpful to imagine that at each time $x$ there is an agent who is making the prediction of $f(x)$.  Each agent can see the past but not the present or future, so the agent at $x$ knows $f \upharpoonright (-\infty, x)$ but no other values of $f$.  Hardin and Taylor make an extensive study of the predictions that can be made by agents in this and similar situations in their book \cite{HTcoord}.

One idea considered by Hardin and Taylor is the possibility that the agents can see the past but don't know what time it is---in other words, the agent at time $x$ knows the value of $f(x-c)$ for every $c > 0$, but doesn't know the value of $x$.  They define a prediction strategy to be \emph{anonymous} if it can be used by agents whose knowledge is restricted in this way.  And they show that even with this restriction on the knowledge of the agents, there is a prediction strategy that guarantees that for every function $f$, all but countably many agents are correct.

In this paper we consider several variations on the notion of anonymity.  What if the agents not only don't know what time it is, they also don't know the rate at which time passes?  What if each agent knows the order of past events, but nothing about the length of time between events? We will see that in some cases the agents are still able to perform well, with only countably many making incorrect predictions, while in other cases they do very badly, with only countably many making correct predictions.

\section{Definitions}

To state our results formally, we will need some definitions.  Let ${}^\mathbb{R}S$ denote the set of all functions from $\mathbb{R}$ to $S$.  Each such function represents a possible time evolution of the system under study.  Following Hardin and Taylor, we call such functions \emph{scenarios}.  Given a scenario $f$, the predictions made by all of the agents could be thought of as determining another scenario, which we will denote $P(f)$.  Thus, $P(f) : \mathbb{R} \to S$, and for any time $x \in \mathbb{R}$, $P(f)(x)$ denotes the prediction of $f(x)$ made by the agent at time $x$ under the scenario $f$.  Since the agent at $x$ can see only the past, scenarios $f$ and $g$ will be indistinguishable to him if $f \upharpoonright (-\infty, x) = g \upharpoonright (-\infty, x)$, and therefore his predictions in these scenarios should be the same.  We therefore make the following definition.

\begin{definition}
A function $P : {}^\mathbb{R}S \to {}^\mathbb{R}S$ is called a \emph{predictor} if it has the property that for all functions $f,g \in {}^\mathbb{R}S$ and all $x \in \mathbb{R}$, if $f \upharpoonright (-\infty, x) = g \upharpoonright (-\infty, x)$ then $P(f)(x) = P(g)(x)$.
\end{definition}

We say that the agent at $x$ \emph{guesses correctly} if his prediction of the value of $f(x)$ is correct---that is, if $P(f)(x) = f(x)$.  Otherwise, agent $x$ \emph{guesses incorrectly}.  We can now state Hardin and Taylor's main result from \cite{HTpeculiar}.

\begin{theorem}[Hardin and Taylor]\label{thm:HT}
There is a predictor $P$ such that for every scenario $f$, only countably many agents guess incorrectly; in other words, $\{x \in \mathbb{R} : P(f)(x) \ne f(x)\}$ is countable.
\end{theorem}

The details of the proof can be found in \cite{HTpeculiar}, but the idea is simple.  Using the axiom of choice, we can let $\prec$ be a well ordering of ${}^\mathbb{R}S$.  Now for any $f \in {}^\mathbb{R}S$ and any $x \in \mathbb{R}$, let $[f]_x = \{g \in {}^\mathbb{R}S : f \upharpoonright (-\infty, x) = g \upharpoonright (-\infty, x)\}$, and let $g^f_x$ be the $\prec$-least element of $[f]_x$.  In other words, $g^f_x$ is the least scenario that is consistent with the information available to the agent at time $x$.  Let $P(f)(x) = g^f_x(x)$.  Hardin and Taylor show that $\{x \in \mathbb{R} : P(f)(x) \ne f(x)\}$ is well ordered (by the usual ordering on $\mathbb{R}$), and any well ordered subset of $\mathbb{R}$ is countable.  (To see why a well ordered set $W \subseteq \mathbb{R}$ must be countable, define a one-to-one function $h : W \to \mathbb{Q}$ by letting $h(x)$ be a rational number larger than $x$ but smaller than the next element of $W$, if there is one.)  Thus, with this predictor, only countably many agents guess incorrectly.

Next we present Hardin and Taylor's definition of anonymity.  Suppose that $f$ is a scenario, and define another scenario $g$ by the equation $g(x) = f(x+b)$, for some constant $b$.  In other words, $g = f \circ t^b$, where $t^b$ is the function $t^b(x) = x+b$.  We will call $t^b$ a \emph{shift function}.  We could think of $g$ as the scenario that represents precisely the same time evolution of the system as $f$, but shifted in time by $b$ units.  Now consider agents at times $x$ and $x+b$ who can see the past, but who don't know where they are located in time.  For any $c>0$, we have
\[
g(x-c) = f(x-c+b) = f(x+b-c).
\]
This means that the past from the point of view of the agent at $x$ in scenario $g$ looks exactly the same as the past from the point of view of the agent at $x+b$ in scenario $f$.  We would therefore expect these agents to make the same predictions.  In other words, we would expect $P(g)(x) = P(f)(x+b)$, or equivalently $P(f \circ t^b)(x) = (P(f) \circ t^b)(x)$.  This motivates Hardin and Taylor's definition of anonymity.

\begin{definition}
A predictor $P$ is \emph{anonymous} if for every scenario $f$ and every real number $b$,
\[
P(f \circ t^b) = P(f) \circ t^b.
\]
\end{definition}

Anonymity of a predictor can be thought of as a sort of invariance under shifts.  This suggests a natural generalization of the notion of anonymity (see \cite[p.\ 82]{HTcoord}).

\begin{definition}
Let $T$ be a family of functions from $\mathbb{R}$ to $\mathbb{R}$.  A predictor $P$ is \emph{$T$-anonymous} if for every scenario $f$ and every $t \in T$,
\[
P(f \circ t) = P(f) \circ t.
\]
\end{definition}

Thus, anonymity is just $T_1$-anonymity, where $T_1$ is the family of shift functions.  But there are other families of functions that we might use instead:
\begin{align*}
T_2 &= \{t : t \text{ is a linear function with positive slope}\},\\
T_3 &= \{t : t \text{ is an infinitely differentiable strictly increasing bijection from $\mathbb{R}$ to $\mathbb{R}$}\},\\
T_4 &= \{t : t \text{ is a strictly increasing bijection from $\mathbb{R}$ to $\mathbb{R}$}\}.
\end{align*}
If $f$ is a scenario and $t \in T_2$, then $f \circ t$ is a scenario that is the same as $f$, except that events have been shifted in time, and also the rate at which events happen has been changed.  Thus, a $T_2$-anonymous predictor is one that can be used by agents who not only don't know what time it is, they also don't know how fast time passes.  Similarly, $T_3$- and $T_4$-anonymous predictors are predictors that can be used by agents who are insensitive to more extreme distortions of time.  Hardin and Taylor ask in \cite{HTcoord} how successful a $T_4$-anonymous predictor can be (see Question 7.8.3 on p.\ 82).

Note that $T_1 \subseteq T_2 \subseteq T_3 \subseteq T_4$, and therefore the requirement of $T_i$-anonymity becomes stronger as $i$ increases.  We will show that there is a $T_2$-anonymous predictor that always performs well (see Theorem \ref{thm:main1}), but there is a scenario in which all $T_3$-anonymous predictors perform badly (see Theorem \ref{thm:main2}).

\section{$T_2$-anonymity}

We need a few definitions and lemmas before we can prove our main result about $T_2$-anonymous predictors.  Recall that a scenario $f$ is said to be \emph{periodic} if there is a number $b \ne 0$ such that for all $x$, $f(x) = f(x+b)$; $b$ is called a \emph{period} of $f$.  Equivalently, $f$ is periodic with period $b$ if $f = f \circ t^b$.  We will need to generalize this concept to functions in $T_2$.

\begin{definition}
We will say that a scenario $f$ is \emph{affine-invariant} if there is some $t \in T_2$ such that $t$ is not the identity function and $f = f \circ t$.
\end{definition}

Since our agents can see only the past, it will also be convenient to have names for scenarios that are periodic or affine-invariant up to some point in time.

\begin{definition}
Suppose that $x$ is a real number and $f$ is a scenario.  If there is some $b \ne 0$ such that $f \upharpoonright (-\infty, x) = (f \circ t^b) \upharpoonright (-\infty, x)$, then we say that $f$ is \emph{past-periodic on $(-\infty, x)$ with period $b$}.  Similarly, if there is some nonidentity function $t \in T_2$ such that $f \upharpoonright (-\infty, x) = (f \circ t) \upharpoonright (-\infty, x)$, then we say that $f$ is \emph{past-affine-invariant on $(-\infty, x)$}.
\end{definition}

Clearly, if a scenario is periodic or affine-invariant, then it is also past-periodic or past-affine-invariant.  In the other direction, if a scenario is past-periodic or past-affine-invariant, then it can be modified to make it periodic or affine-invariant, as our next two lemmas show.

\begin{lemma}\label{lem:perext}
Suppose the scenario $f$ is past-periodic on $(-\infty, x)$. Then there is a periodic scenario $g$ with the property that $f \upharpoonright (-\infty, x) = g \upharpoonright (-\infty, x)$.
\end{lemma}
\begin{proof}
Choose $b \ne 0$ so that $f$ is past-periodic on $(-\infty, x)$ with period $b$, and define an equivalence relation $\sim_b$ on $\mathbb{R}$ by 
\begin{equation*}
y \sim_b z \Leftrightarrow \text{for some integer $n$, } y + nb = z.
\end{equation*}
Let $[y]_b$ denote the equivalence class of $y$ under $\sim_b$.  Since $f$ is past-periodic on $(-\infty, x)$ with period $b$, it is constant on $[y]_b \cap (-\infty, x)$, for every real number $y$.  To define $g$, we simply let the value of $g$ on any equivalence class $[y]_b$ be the constant value of $f$ on $[y]_b \cap (-\infty, x)$.  It is easy to see that $g$ agrees with $f$ on $(-\infty, x)$ and $g$ is periodic.
\end{proof}

\begin{lemma}\label{lem:afext}
Suppose the scenario $f$ is past-affine-invariant on $(-\infty, x)$. Then there is a scenario $g$ that is affine-invariant and has the property that $f \upharpoonright (-\infty, x) = g \upharpoonright (-\infty, x)$. 
\end{lemma}
\begin{proof}
Since $f$ is past-affine-invariant on $(-\infty, x)$, there is some nonidentity function $t \in T_2$ such that $f \upharpoonright (-\infty, x) = (f \circ t) \upharpoonright (-\infty, x)$.  For any positive integer $n$, let $t^n$ denote the $n$-fold composition of $t$ with itself.  Similarly, let $t^{-n}$ denote the $n$-fold composition of $t^{-1}$ with itself, and let $t^0$ be the identity function.  Now if we define $y \sim_t z$ to mean that there is some integer $n$ such that $t^n(y) = z$, then $\sim_t$ is an equivalence relation on $\mathbb{R}$.  The rest of the proof is similar to the proof of Lemma \ref{lem:perext}: $f$ is constant on the intersection of each equivalence class with $(-\infty, x)$, and we can define $g$ to agree with $f$ on $(-\infty, x)$ and to be constant on all equivalence classes.
\end{proof}

We will also need to know that the periods in periodicity and past-periodicity of a function always match.

\begin{lemma}\label{lem:permatch}
Suppose $f$ is past-periodic on $(-\infty, x)$ with period $b$.  If $f$ is also periodic, then it has period $b$.
\end{lemma}
\begin{proof}
Suppose $f$ is periodic with period $c$.  To see that $b$ is also a period of $f$, consider any real number $y$.  Choose an integer $n$ such that $y + nc$ is less than $x$.  Then
\begin{align*}
f(y) &= f(y + nc) & & \text{(since $f$ has period $c$)}\\
&= f(y + nc + b) & & \text{(since $f$ is past-periodic on $(-\infty, x)$ with period $b$)}\\
&= f(y + b) & & \text{(since $f$ has period $c$).} \qedhere
\end{align*}
\end{proof}

We are now ready to prove that there is a $T_2$-anonymous predictor that performs very well.

\begin{theorem}\label{thm:main1}
There is a $T_2$-anonymous predictor $P$ such that for every scenario $f$, $\{x \in \mathbb{R} : P(f)(x) \ne f(x)\}$ is countable.  In other words, only countably many agents guess incorrectly.
\end{theorem}

\begin{proof}
The predictor we use is similar to the one described in \cite{HTpeculiar}. Using the axiom of choice, fix a well ordering $\prec$ of $^{\mathbb{R}}S$ that lists periodic scenarios first, then other affine-invariant scenarios, and then the rest of the scenarios. Suppose $f$ is an arbitrary scenario and $x$ is an arbitrary agent. Let $[f]_x = \{ g \in {}^{\mathbb{R}} S : f \upharpoonright (-\infty, x) = (g \circ t) \upharpoonright (-\infty, x) \text{ for some $t \in T_2$}\}$. Let $g^f_x$ denote the $\prec$-least element of $[f]_x$.  To define $P(f)(x)$, choose any $t \in T_2$ for which $f \upharpoonright (-\infty, x) = (g^f_x \circ t) \upharpoonright (-\infty, x)$, and let
\begin{equation}\label{eq:defP}
P(f)(x) = (g^f_x \circ t) (x).
\end{equation}

We first need to show that $P$ is well defined---that is, that the value of $P(f)(x)$ doesn't depend on the choice of $t$. Suppose that for some scenario $f$ and some agent $x$, there are two distinct functions $t_1, t_2 \in T_2$ such that
\begin{equation}\label{eq:first}
f \upharpoonright (-\infty, x) = (g^f_x \circ t_1) \upharpoonright (-\infty, x) = (g^f_x \circ t_2) \upharpoonright(-\infty, x).
\end{equation}
In order to prove that $P(f)(x)$ is well defined, we need to show that $(g^f_x \circ t_1) (x) = (g^f_x \circ t_2) (x)$.  In other words, letting $x_1 = t_1(x)$ and $x_2 = t_2(x)$, we must show that $g^f_x(x_1) = g^f_x(x_2)$.  Of course, if $x_1 = x_2$ then this clearly holds, so we may assume that $x_1 \ne x_2$.

Let $\bar{t} = t_2 \circ t_1 ^{-1}$.  Then $\bar{t}(x_1) = x_2$, and it follows from \eqref{eq:first} that
\begin{equation}\label{eq:x1}
g^f _x \upharpoonright (-\infty, x_1) = (g^f _x \circ \bar{t}) \upharpoonright (-\infty, x_1)
\end{equation}
and
\begin{equation}\label{eq:x2}
g^f _x \upharpoonright (-\infty, x_2) = (g^f _x \circ \bar{t}^{-1}) \upharpoonright (-\infty, x_2).
\end{equation}
In particular, \eqref{eq:x1} reveals that $g^f_x$ is past-affine-invariant on $(-\infty, x_1)$, so Lemma \ref{lem:afext} tells us that there is an affine-invariant function $g$ that agrees with $g^f_x$ on $(-\infty, x_1)$.  Therefore $g \circ t_1$ agrees with $f$ on $(-\infty, x)$, so $g$ is an element of $[f]_x$. Since $g^f_x$ is the $\prec$-least element of $[f]_x$, it follows that $g^f_x \preccurlyeq g$. Since $\prec$ lists affine-invariant scenarios first, and $g$ is affine-invariant, we deduce that $g^f _x$ is affine-invariant as well. So we may fix an $s$ in $T_2$ with the property that $s$ is not the identity function and
\begin{equation}\label{eq:s}
g^f _x = g^f _x \circ s.
\end{equation}
An immediate consequence of this is that $g^f _x = g^f _x \circ s^{-1}$.  By replacing $s$ with $s^{-1}$ if necessary, we may assume that $s(x_1) \le x_1$.

To prove that $g^f_x(x_1) = g^f_x(x_2)$, we now consider two cases.

Case 1: $s$ and $\bar{t}$ commute.  If $s(x_1) = x_1$, then
\[
s(x_2) = s(\bar{t}(x_1)) = \bar{t}(s(x_1)) = \bar{t}(x_1) = x_2.
\]
Since $x_1 \neq x_2$, this means that the linear function $s$ has two distinct fixed points---a contradiction, since $s$ is not the identity function. Hence $s(x_1) \neq x_1$, and since we have $s(x_1) \le x_1$, we conclude that $s(x_1) < x_1$.  Therefore
\begin{align*}
g^f_x(x_1) &= (g^f_x \circ s) (x_1)= g^f_x (s(x_1)) & & \text{(by \eqref{eq:s})}\\
&= (g^f_x \circ \bar{t}) (s(x_1)) = g^f_x(\bar{t}(s(x_1))) & & \text{(by \eqref{eq:x1}, since $s(x_1) < x_1$)}\\
&= g^f_x (s(\bar{t}(x_1))) & & \text{(since $s$ and $\bar{t}$ commute)}\\
&= (g^f_x \circ s)(x_2) & & \text{(since $\bar{t}(x_1) = x_2$)}\\
&= g^f_x(x_2) & & \text{(by \eqref{eq:s})},
\end{align*}
as required.

Case 2: $s$ and $\bar{t}$ do not commute.  Consider any $y < x_2$. In what follows, we abbreviate expressions like $s(\bar{t}^{-1}(y))$ to $s\bar{t}^{-1}(y)$ to prevent parentheses from stacking up. Note that $y < x_2$ implies that $\bar{t}^{-1}(y) < \bar{t}^{-1}(x_2) = x_1$, and therefore $s \bar{t}^{-1}(y) < s(x_1) \le x_1$.  Thus,
\begin{equation}\label{eq:commutator}
\begin{aligned}
g^f_x (y) &= (g^f_x \circ \bar{t}^{-1}) (y) = g^f_x (\bar{t}^{-1}(y)) & &\text{(by \eqref{eq:x2}, since $y <x_2$)} \\
&= (g^f_x \circ s) (\bar{t}^{-1}(y)) = g^f_x (s \bar{t}^{-1}(y)) & &\text{(by \eqref{eq:s})} \\
&= (g^f_x \circ \bar{t})(s \bar{t}^{-1}(y)) = g^f_x ( \bar{t}s \bar{t}^{-1}(y)) & &\text{(by \eqref{eq:x1}, since $s\bar{t}^{-1}(y) <x_1$)} \\
&= (g^f_x \circ s^{-1})( \bar{t}s\bar{t}^{-1}(y)) & &\text{(by \eqref{eq:s})}\\
&= g^f_x(s^{-1}\bar{t}s \bar{t}^{-1}(y)).
\end{aligned}
\end{equation}

We let the reader verify that $s^{-1}\bar{t}s \bar{t}^{-1}$ is a shift function, and since $s$ and $\bar{t}$ don't commute, it is not the identity function. Therefore we can fix $b \ne 0$ so that $s^{-1}\bar{t}s \bar{t}^{-1} = t^b$.  Equation \eqref{eq:commutator} now tells us that $g^f_x \upharpoonright (-\infty, x_2) = (g^f_x \circ t^b) \upharpoonright (-\infty, x_2)$; in other words, $g^f_x$ is past-periodic on $(-\infty, x_2)$ with period $b$. By Lemma \ref{lem:perext} we know that there is a periodic scenario $g$ that agrees with $g^f_x$ on $(-\infty, x_2)$. Fix such a $g$. Then $g \circ t_2$ agrees with $f$ on $(-\infty, x)$, and therefore $g$ is an element of $[f]_x$. Hence, $g^f_x \preccurlyeq g$. Since $g$ is periodic, and $\prec$ lists periodic scenarios first, $g^f_x$ must also be periodic. By Lemma \ref{lem:permatch}, $g^f_x$ has period $b$. Putting this all together, we have shown that
\begin{equation}\label{eq:tb}
g^f_x = g^f_x \circ t^b.
\end{equation}

Replacing $b$ with $-b$ if necessary, we may assume that $b < 0$, so that
$t^b(y) < y$ for all $y$ in $\mathbb{R}$.  Therefore for every $y < x_1$,
$t^b \bar{t}(y) < \bar{t}(y) < \bar{t}(x_1) = x_2$, so
\begin{equation}\label{eq:t-1tbt}
\begin{aligned}
g^f_x(y) &= (g^f_x \circ \bar{t})(y) = g^f_x(\bar{t}(y)) & & \text{(by
\eqref{eq:x1}, since $y < x_1$)}\\
&= (g^f_x \circ t^b)(\bar{t}(y)) = g^f_x(t^b \bar{t}(y)) & & \text{(by
\eqref{eq:tb})}\\
&= (g^f_x \circ \bar{t}^{-1})(t^b \bar{t}(y)) = g^f_x(\bar{t}^{-1} t^b
\bar{t} (y)) & & \text{(by \eqref{eq:x2}, since $t^b \bar{t}(y) < x_2$).}
\end{aligned}
\end{equation}
It is easy to verify that $\bar{t}^{-1} t^b \bar{t}$ is also a nonidentity
shift function, so we can choose some $c \ne 0$ such that $\bar{t}^{-1}
t^b \bar{t} = t^c$.  Thus, equation \eqref{eq:t-1tbt} says that $g^f_x
\upharpoonright (-\infty, x_1) = (g^f_x \circ t^c) \upharpoonright
(-\infty, x_1)$.  In other words, $g^f_x$ is past-periodic on $(-\infty,
x_1)$ with period $c$.  As before, we can use Lemma \ref{lem:permatch} to
deduce that
\begin{equation}\label{eq:tc}
g^f_x = g^f_x \circ t^c.
\end{equation}

To complete our proof that $P$ is well defined, we can now compute
\begin{align*}
g^f_x(x_2) &= (g^f_x \circ t^b)(x_2) = g^f_x(t^b(x_2)) & & \text{(by
\eqref{eq:tb})}\\
&= (g^f_x \circ \bar{t}^{-1})(t^b(x_2)) = g^f_x(\bar{t}^{-1} t^b(x_2)) & &
\text{(by \eqref{eq:x2}, since $t^b(x_2) < x_2$)}\\
&= g^f_x(\bar{t}^{-1} t^b \bar{t}(x_1)) & & \text{(since $x_2 =
\bar{t}(x_1)$)}\\
&= (g^f_x \circ t^c)(x_1) & & \text{(since $\bar{t}^{-1} t^b \bar{t} =
t^c$)}\\
&= g^f_x(x_1) & & \text{(by \eqref{eq:tc}).}
\end{align*}

We must now show that $P$ is $T_2$-anonymous. Given a scenario $f$ and an element $t$ of $T_2$, we need to show that $P(f \circ t) = P(f)\circ t$. Fix an arbitrary agent $x$. We first claim that $[f \circ t]_x = [f]_{t(x)}$.  To see why, suppose $g$ is an element of $[f \circ t]_x$. Fix $t_g \in T_2$ so that $(f \circ t) \upharpoonright (-\infty, x) = (g \circ t_g) \upharpoonright (-\infty, x)$.  Then $f \upharpoonright (-\infty, t(x)) = (g \circ (t_g \circ t^{-1})) \upharpoonright (-\infty, t(x))$.  Since $t_g \circ t^{-1} \in T_2$, it follows that $g$ is in $[f]_{t(x)}$. The other inclusion has an almost identical proof. The fact that $[f \circ t]_x = [f]_{t(x)}$ implies that the two sets have the same $\prec$-least element: 
\begin{equation}\label{eq:sameguess}
g^{f \circ t}_x = g^f _{t(x)}. 
\end{equation}

Since $g^f_{t(x)}$ is in $[f]_{t(x)}$, we can choose some $t' \in T_2$ such that $f \upharpoonright (-\infty, t(x)) = (g^f_{t(x)} \circ t') \upharpoonright (-\infty, t(x))$, and therefore
\[
(f \circ t) \upharpoonright (-\infty, x) = (g^f_{t(x)} \circ t' \circ t) \upharpoonright (-\infty, x) = (g^{f \circ t}_x \circ (t' \circ t)) \upharpoonright (-\infty, x).
\]
Applying the definition of $P$ (equation \eqref{eq:defP}), we have
\[
P(f \circ t)(x) = (g^{f \circ t}_x \circ (t' \circ t))(x) = (g^f_{t(x)} \circ t')(t(x)) = (P(f) \circ t)(x).
\]
Since $x$ was arbitrary, it follows that $P(f \circ t)= P(f) \circ t$.

Finally, we must prove that with the predictor $P$, only countably many agents guess incorrectly.  Let $W$ be the set of agents who guess incorrectly; that is, $W = \{x \in \mathbb{R} : P(f)(x) \ne f(x)\}$.  Suppose that $x$ and $y$ are both in $W$ and $x<y$. We claim that $g^f_x \prec g^f _y$.  To prove this claim, we first show that $[f]_y \subseteq [f]_x$. Consider an arbitrary $g$ in $[f]_y$. Fix a $t$ in $T_2$ so that $f$ agrees with $g \circ t$ on $(-\infty, y)$. Since $x < y$, it follows trivially that $f$ agrees with $g \circ t$ on $(-\infty, x)$; hence, $g$ is in $[f]_x$.

Since $[f]_y \subseteq [f]_x$, the $\prec$-least element of $[f]_y$ greater than or equal to the $\prec$-least element of $[f]_x$. In other words, $g^f_x \preccurlyeq g^f_y$. We now claim that $g^f_x \neq g^f_y$. Suppose, toward contradiction, that $g^f_x = g^f_y$.  Choose a $t$ in $T_2$ such that $f$ agrees with $g^f_y \circ t = g^f_x \circ t$ on $(-\infty, y)$.  Then by the definition of $P$,
\[
P(f)(x) = (g^f_x \circ t)(x) = (g^f_y \circ t)(x) = f(x),
\]
where the last equal sign follows from the fact that $f$ and $g^f_y \circ t$ agree on $(-\infty, y)$ and $x < y $. However, this is a contradiction, since we assumed that $x$ guesses incorrectly. So $g^f_x \neq g^f_y$. Combining $g^f_x \preccurlyeq g^f_y$ and $g^f_x \neq g^f_y$, we have $g^f_y \prec g^f_y$.

Now, suppose $W$ has an infinite descending chain $\cdots < w_2 < w_1 < w_0$. By the claim we have just proven, this gives us an infinite descending chain $\cdots \prec g^f_{w_2} \prec g^f_{w_1} \prec g^f_{w_0}$ in ${}^{\mathbb{R}}S$, contradicting the fact that $\prec$ is a well ordering. Thus $W$ has no infinite descending chain, so it is well ordered, and therefore countable.
\end{proof}

\section{$T_3$-anonymity}
In this section, we construct a scenario in which agents using \emph{any} $T_3$-anonymous predictor perform poorly. Broadly speaking, this scenario ensures that every agent sees the same thing looking into the past, up to distortion by elements in $T_3$. The condition of $T_3$-anonymity then implies that every agent guesses the same thing.  The scenario takes each of its values only countably many times, so only countably many agents guess correctly.  We spend the rest of the section working out the details of this argument.

Central to the construction of this scenario will be functions that are \emph{smooth} (that is, infinitely differentiable at each point in their domain), but fail to be analytic at certain points.  For example the function
\[
h(x)=
  \begin{cases} 
      0, & \text{if $x \leq 0$,} \\
      e^{-1/x}, & \text{if $ x > 0$} \\
  \end{cases}
\]
is smooth everywhere, but fails to be analytic at 0.  Note in particular that for every positive integer $k$, $h^{(k)}(0) = 0$, where $h^{(k)}$ denotes the $k$th derivative of $h$.

We can use $h$ to define a \emph{smooth transition function} $s:[0,1] \to [0,1]$ as follows:
\begin{equation}\label{eq:sx}
s(x) = \frac{h(x)}{h(x) + h(1-x)}.
\end{equation}
The graph of $s$ transitions smoothly from being very flat at the point $(0,0)$ to being very flat at the point $(1,1)$ (see Figure \ref{fig:s}). More precisely, we have $s(0) = 0$, $s(1) = 1$, $s$ is increasing and smooth on $[0,1]$, and for every positive integer $k$, $s_+^{(k)}(0) = s_-^{(k)}(1) = 0$.  (The subscripts $+$ and $-$ here indicate that these are one-sided derivatives.)

\begin{figure}
\includegraphics{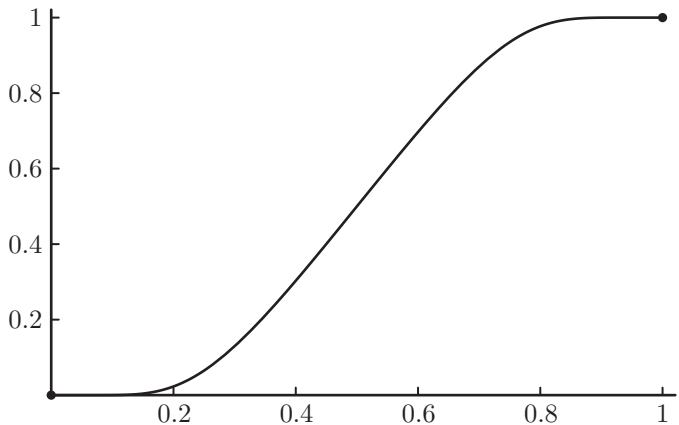}
\caption{The graph of $s$.}\label{fig:s}
\end{figure}

By shifting and rescaling $s$, we can define similar smooth transition functions connecting other points.  Let $A = (p_1, q_1)$ and $B=(p_2, q_2)$ be two points in $\mathbb{R}^2$, and suppose that $p_1 < p_2$ and $q_1 < q_2$; we say that $B$ is \emph{above and to the right of} $A$. We define the function $s_{AB}: [p_1, p_2] \to [q_1, q_2]$ by the equation
\[
s_{AB}(x) = (q_2 - q_1) \cdot s\!\left(\frac{x - p_1}{p_2 - p_1}\right) + q_1.
\]
The graph of $s_{AB}$ passes through the points $A$ and $B$, and $s_{AB}$ is increasing and smooth on $[p_1, p_2]$. When we use the notation $s_{AB}$, we implicitly assume that $B$ is above and to the right of $A$. Notice that since $s_{AB}$ is increasing, it is invertible. Also, by the chain rule, we have
\[
(s_{AB})'(x) = \frac{q_2-q_1}{p_2-p_1} \cdot s'\!\left(\frac{x-p_1}{p_2-p_1}\right),
\]
and using the chain rule $k$ times, we have
\begin{equation}\label{eq:sAB(k)}
(s_{AB})^{(k)}(x) = \frac{q_2-q_1}{(p_2-p_1)^k} \cdot s^{(k)}\!\left(\frac{x-p_1}{p_2-p_1}\right).
\end{equation}
In particular, at the endpoints of the domain we have the one-sided derivatives $(s_{AB})_+^{(k)}(p_1) = (s_{AB})_-^{(k)}(p_2) = 0$.  Thus, $s_{AB}$ is very flat at $A$ and $B$.

Before continuing, we note one reason that these smooth transitions functions will be useful to us.  Suppose that $A$, $B$, and $C$ are three points in the plane, with $C$ above and to the right of $B$, which is above and to the right of $A$.  If we concatenate $s_{AB}$ and $s_{BC}$ (in the obvious way), we obtain a smooth increasing function passing through all three of the points $A$, $B$, and $C$.  Repeating this reasoning, we can create a smooth increasing function passing through a sequence of points going up and to the right in $\mathbb{R}^2$.

We now use our smooth transition functions to define an equivalence relation on $\mathbb{R}$.  We will say that a point $(p,q)$ in $\mathbb{R}^2$ is a \emph{rational pair} if both $p$ and $q$ are rational numbers. Since the set of rational pairs is countable, the set
\[
F = \{s_{AB} : \text{$A$ and $B$ are both rational pairs}\}
\]
is countable as well.  For real numbers $x$ and $y$, we define $x \sim_F y$ to mean that there is some $f \in F$ such that either $f(x) = y$ or $f^{-1}(x) = y$.  We will say that a finite sequence of numbers $(x_1, x_2, \ldots, x_n)$ is an \emph{$F$-path of length $n$} from $x$ to $y$ if
\[
x = x_1 \sim_F x_2 \sim_F \cdots \sim_F x_n = y.
\]
Now we let $x \sim_F^* y$ mean that there is an $F$-path from $x$ to $y$.  For example, if $x \sim_F y$, then $(x, y)$ is an $F$-path of length 2 from $x$ to $y$, so $x \sim_F^* y$.  For any real number $x$, the one-term sequence $(x)$ is an $F$-path of length 1 from $x$ to $x$, and therefore $x \sim_F^* x$; in other words, $\sim_F^*$ is reflexive.

\begin{lemma}
The relation $\sim_F^*$ is an equivalence relation on $\mathbb{R}$.  Furthermore, for any real number $x$, the equivalence class $[x]_{\sim_F^*}$ is countable. 
\end{lemma}
\begin{proof}
We have already seen that $\sim_F^*$ is reflexive.  Clearly $\sim_F$ is symmetric.  It follows that the reverse of an $F$-path is another $F$-path, and therefore $\sim_F^*$ is symmetric.  Transitivity of $\sim_F^*$ can be proven by concatenating $F$-paths.  To see that the equivalence classes are countable, observe that since $F$ is countable, for any $x \in \mathbb{R}$ there are only countably many $y$ such that $x \sim_F y$.  It follows that there are only countably many $F$-paths starting at $x$. 
\end{proof}

Next we state and prove an interesting lemma from analysis. This lemma will be one of the key ingredients in the proof of the main result of this section. 

\begin{lemma}\label{lem:smooth}
Suppose $f$ is smooth on $(-\infty, w)$ and continuous on $(-\infty, w]$. Suppose also that $\lim_{x \to w^{-}}f^{(k)}(x) = 0$ for every positive integer $k$. Then $f_-^{(k)}(w) = 0$ for all $k$. 
\end{lemma}
\begin{proof}
For each $x < w$, use the mean value theorem to fix a $c_x$ so that $f(w) - f(x) = f'(c_x)(w-x)$. Rearranging this equation and taking the limit of both sides yields
\[
\lim_{x \to w^{-}}\frac{f(w) - f(x)}{w-x} = \lim_{x \to w^{-}}f'(c_x).
\]
The left side of this equation is the definition of $f_-'(w)$, and the right side equals $0$ by hypothesis, so $f_-'(w) = 0$. But now the same reasoning can be applied to $f'$, yielding $f_-''(w) = 0$.  Since we can continue to repeat this reasoning, the lemma follows by induction. 
\end{proof}

We are now ready to state and prove the main result of this section.

\begin{theorem}\label{thm:main2}
There is a set $S$ and a scenario $f \in {}^{\mathbb{R}}S$ such that for every $T_3$-anonymous predictor $P$, $\{x \in \mathbb{R} : P(f)(x) = f(x)\}$ is countable.  In other words, only countably many agents guess correctly.
\end{theorem}
\begin{proof}
Let $S = \mathbb{R}/{\sim_F^*} = \{[x]_{\sim_F^*} : x \in \mathbb{R} \}$, and let $f : \mathbb{R} \to S$ be defined by the equation
\[
f(x) = [x]_{\sim_F^*}.
\]
Thus, the state of the scenario $f$ at time $x$ is $[x]_{\sim_F^*}$. Suppose $P$ is a $T_3$-anonymous predictor.  We claim that using this predictor, all agents make the same guess in the scenario $f$.  Assuming this claim, we can let $[x]_{\sim_F^*} \in S$ be the common guess made by all agents.  Then for every $y \in \mathbb{R}$,
\begin{align*}
y \text{ guesses correctly} &\Leftrightarrow P(f)(y) = f(y)\\
&\Leftrightarrow [x]_{\sim_F^*} = [y]_{\sim_F^*}\\
&\Leftrightarrow y \in [x]_{\sim_F^*}.
\end{align*}
In other words, the countably many agents in the equivalence class $[x]_{\sim_F^*}$ guess correctly, and everyone else guesses incorrectly.  Thus, proving this claim will suffice to prove the theorem.

To show that all agents make the same guess, let $w$ and $z$ be arbitrary agents. We will prove that $P(f)(w) = P(f)(z)$ by showing that there is a function $t : \mathbb{R} \to \mathbb{R}$ with the following properties:
\begin{itemize}
\item $t \in T_3$
\item $t(w) = z$
\item $f \upharpoonright (-\infty, w) = (f \circ t) \upharpoonright (-\infty, w)$.
\end{itemize}
If we can find a function $t$ with these properties, then we can conclude that
\begin{align*}
P(f)(w) &= P(f \circ t)(w) & & \text{(since $f$ agrees with $f \circ t$ on $(-\infty, w)$)}\\
&= (P(f) \circ t)(w) & & \text{(since $t$ is in $T_3$, and $P$ is $T_3$-anonymous)}\\
&= P(f)(z) & & \text{(since $t(w) = z$).}
\end{align*}

To construct $t$, let $(p_i)$ be an increasing sequence of rational numbers converging to $w$ with $p_0 > w - 1$.  Let $(q_i)$ to be an increasing sequence of rationals converging to $z$ with the property that for every $i > 0$,
\[
z - q_i  < \frac{(p_{i+1} - p_i)^i}{i}.
\]
Notice that for every $i > 0$ we have $p_{i+1} - p_i < w - p_0 < 1$ and
\[
q_{i+1} - q_i < z - q_i < \frac{(p_{i+1} - p_i)^i}{i},
\]
and therefore
\begin{equation}\label{eq:pqbounds}
\frac{q_{i+1} - q_i}{(p_{i+1} - p_i)^i} < \frac{1}{i}.
\end{equation}

For each negative integer $i$, let $p_i = p_0 + i$ and $q_i = q_0 + i$.  For every integer $i$, let $A_i = (p_i, q_i)$, and for $i \ge 0$ let $B_i = (w+i, z+i)$.
We claim that the following function $t$ has the properties that we want:
\[
t(x)=
  \begin{cases} 
     s_{A_i A_{i+1}}(x), & \text{if $x \in [p_i, p_{i+1}]$,} \\
     s_{B_i B_{i+1}}(x), & \text{if $ x \in [w+i, w+i+1]$, where $i \ge 0$.} \\
  \end{cases}
\]

It is not hard to see that $t$ is an increasing bijection from $\mathbb{R}$ to $\mathbb{R}$; therefore it is continuous on $\mathbb{R}$.  By our earlier comment about concatenating smooth transition functions, $t$ is smooth everywhere except perhaps at $w$.  It is also clear that for every positive integer $k$, $t_+^{(k)}(w) = (s_{B_0B_1})_+^{(k)}(w) = 0$.  If we can show that $t_-^{(k)}(w) = 0$ as well, then it will follow that $t$ is smooth at $w$, and therefore $t \in T_3$.  To do this, by Lemma \ref{lem:smooth}, it will suffice to show that $\lim_{x \to w^-} t^{(k)}(x) = 0$.

Fix a positive integer $k$, and suppose $i \ge k$.  According to \eqref{eq:sAB(k)},
\[
(s_{A_iA_{i+1}})^{(k)}(x) = \frac{q_{i+1}-q_i}{(p_{i+1}-p_i)^k} \cdot s^{(k)}\!\left(\frac{x-p_i}{p_{i+1}-p_i}\right),
\]
where $s$ is the function defined in \eqref{eq:sx}. Since $s$ is smooth, each of its derivatives is continuous. Using the extreme value theorem, we let $M_k$ be the maximum value of $|s^{(k)}(x)|$ for $x \in [0,1]$. Then for each $x$ in $[p_i, p_{i+1}]$, we have
\begin{equation}\label{eq:sAiAi+1(k)}
\left|(s_{A_iA_{i+1}})^{(k)}(x)\right| \leq \frac{q_{i+1}-q_i}{(p_{i+1}-p_i)^k} \cdot M_k.
\end{equation}
Using the inequalities $p_{i+1} - p_i < 1$ and $i \ge k$, and then applying $\eqref{eq:pqbounds}$, we have
\begin{equation}\label{eq:coeffbounds}
0 < \frac{q_{i+1} - q_i}{(p_{i+1} - p_i)^k} \le \frac{q_{i+1} - q_i}{(p_{i+1} - p_i)^i} < \frac{1}{i}. 
\end{equation}
Combining \eqref{eq:sAiAi+1(k)} and \eqref{eq:coeffbounds}, we conclude that if $i \ge k$ and $x \in [p_i, p_{i+1}]$, then
\begin{equation}\label{eq:sAiAi+1bound}
\left|(s_{A_iA_{i+1}})^{(k)}(x)\right| < \frac{M_k}{i}.
\end{equation}

Now, to show that $\lim_{x \to w^-} t^{(k)}(x) = 0$, suppose $\epsilon > 0$. Fix $I \ge k$ large enough that $M_k/I < \epsilon$.  Now suppose $p_I < x < w$.  Then $x \in [p_i,p_{i+1}]$ for some $i \ge I$. By \eqref{eq:sAiAi+1bound}, 
\[
\left|t^{(k)}(x)\right| = \left|(s_{A_iA_{i+1}})^{(k)}(x)\right| < \frac{M_k}{i} \leq \frac{M_k}{I}< \epsilon,
\]
as required.  Therefore, $t$ is an element of $T_3$.

Since the graph of $t$ passes through $B_0 = (w,z)$, $t(w) = z$.  All that remains is to show that $f$ agrees with $f \circ t$ on $(-\infty, w)$. Suppose $x < w$.  Then $x \in [p_n, p_{n+1}]$ for some integer $n$, and $t(x) = s_{A_nA_{n+1}}(x)$.  Note that $s_{A_n A_{n+1}}(x) \sim_F^* x$,  since $s_{A_n A_{n+1}} \in F$. Therefore
\begin{align*}
(f \circ t) (x) &= f(t(x)) = f(s_{A_n A_{n+1}}(x)) & &\text{(by the definition of $t$)} \\
&= [s_{A_n A_{n+1}}(x)]_{\sim_F^*} & & \text{(by the definition of $f$)}\\
&= [x]_{\sim_F^*} & & \text{(since $s_{A_n A_{n+1}}(x) \sim_F^* x$)}\\
& = f(x) & & \text{(by the definition of $f$).}
\end{align*}

Thus, $t$ satisfies the three conditions we needed.  This completes the proof of our claim that all agents make the same guess, and therefore the proof of the theorem.
\end{proof}

There is plenty to discover about $T$-anonymous predictors when $T_2 \subseteq T \subseteq T_3$. In particular, what if $T$ consists of increasing polynomials or, alternatively, increasing analytic functions? What if $T$ consists of increasing computable functions? At what point do agents start to perform poorly?


\begin{thebibliography}{9}
\bibitem{HTpeculiar} Hardin,~C.~S.\ and Taylor,~A.~D., A peculiar connection between the axiom of choice and predicting the future, \textit{Amer.\ Math.\ Monthly} \textbf{115} (2008), pp.\ 91--96.
\bibitem{HTcoord} Hardin,~C.~S.\ and Taylor,~A.~D., \textit{The Mathematics of Coordinated Inference: A Study of Generalized Hat Problems}, Springer International Publishing, Cham, Switzerland, 2013.
\end{thebibliography}
\end{document}